\documentclass[11pt,reqno]{amsart}
\usepackage[body={7in,9in},centering]{geometry}
\usepackage{amsmath,amssymb,amsthm,mathrsfs}
\usepackage{color,xcolor}
\definecolor{cobalt}{RGB}{61,99,181}
\usepackage[colorlinks,citecolor=cobalt,linkcolor=cobalt]{hyperref}
\usepackage{float}
\usepackage{exscale}
\usepackage{relsize}
\usepackage{graphicx}
\usepackage{tikz}
\newtheorem{thm}{Theorem}[section]
\newtheorem{defi}[thm]{Definition}
\newtheorem{cor}[thm]{Corollary}
\newtheorem{lem}[thm]{Lemma}

\newtheorem{ex}[thm]{Example}

\numberwithin{equation}{section}

\date{\today}

\makeatletter

\newcommand{\Rmnum}[1]{\expandafter\@slowromancap\romannumeral #1@}
\makeatother

\newcommand{\C}{\mathbb{C}}

\begin{document}

\title[Composition operators on $L^{p}_{\lambda}$]{Composition operators on weighted $L^{p}$ spaces of a tree}
\author[Han Xu]{Han Xu}
\address{
\textsuperscript{1}
College of Mathematics and Statistics, Chongqing University, Chongqing, 401331, P. R. China}
\email{xuh2022@cqu.edu.cn}

\author[Xiaoyan Zhang]{Xiaoyan Zhang}
\address{
\textsuperscript{2}
 College of Mathematics and Statistics, Chongqing University, Chongqing, 401331, P. R. China}
\email{15282233835@163.com}



\begin{abstract}
In this paper, we study the boundedness, compactness and Schatten class membership of   composition operators on the weighted $L^{p}$-space of a tree $L^{p}_{\lambda}(T)$ with $1\leq p <\infty$.
\end{abstract} \maketitle

\section{Introduction}\label{S1}

Let $S$ be a non-empty set and $X$ be a Banach space consisting of complex-valued functions on $S$, given a self map $\varphi$ of $S$, we define a linear operator $C_{\varphi}$ by $$C_{\varphi}f(x)=f(\varphi(x))$$ for all $x\in S$ and $f\in X$, then this linear operator is called a composition operator with the symbol $\varphi$ on $X$.\

Composition operators have been well-studied in classical spaces of analytic functions, and many books and articles relevant to this subject are available. In \cite{Al15}, composition operators on the Bergman space $A^{p}$ and the Hardy space $H^{p}$ of the unit disc $\mathbb{D}$ on the complex plane have been discussed in detail . Recently, there are many relevant researches of composition operators on spaces derived from these classical function spaces, such as the study of composition operators on Fock-type space in \cite{Al1} and composition operators on weighted Dirichlet spaces $D_{\alpha}^{p}$ in \cite{Al2}. Also, there are excellent books about composition operators which interested readers can refer to \cite{Al3, Al4, Al6}.\

In recent years, the study of composition operators has got into a wider realm. In \cite{Al8}, motivated by the interest in the study of operators on discrete structures, Colonna and Easley firstly defined the Lipschitz spaces of an infinite rooted tree, which can be viewed as a complete metric space with respect to the edge-countable metric. As defined, this space is the natural discrete analogue of the Bloch space, and Colonna et al. studied composition operators on this space in \cite{Al5}. Operators on this space have been furthered studied and detailed information can be founded in \cite{Al7}. More recently, Muthukumar and Ponnusamy defined a discrete analogue of generalized Hardy space on homogenous trees in \cite{Al9} and carried out the study of composition operators on this discrete analogue of generalized Hardy space in \cite{Al10}.\

Besides the discrete analogue of Bloch space and generalized Hardy space, there are other natural spaces of functions on an infinite tree. The first is the space of bounded functions, denoted $L^{\infty}$, which is a Banach space and can be viewed as the discrete analogue to the space of bounded analytic functions $H^{\infty}$. The second is the space of functions on trees which is $p$-summable ($1\leq p<\infty$), denoted $L^{p}$. This space is also a Banach space and can be regarded as the discrete analogue of Bergman space $A^{p}$ on the unit disk $\mathbb{D}$. Both two spaces can be extended into weighted Banach spaces, denoted by $L^{\infty}_{\lambda}$ and $L^{p}_{\lambda}$ , respectively. For the weighted Banach spaces $L^{\infty}_{\lambda}$,  multiplication operators and composition operators have been studied in \cite{Al11} and \cite{Al12}, respectively. In \cite{Al14}, shift operators on weighted $L^{p}$ spaces of directed trees has been studied, which can be viewed as special composition operators on a discrete structure.\

In this paper, we mainly investigate the elementary properties of composition operators on a given weighted $L^{p}$ space of a tree (denoted by $L^{p}_{\lambda}(T)$). We associate the properties of composition operators to the properties of weight function and the symbol. These properties, such as boundedness, compactness, or being an isometry, ect, occupy an important position in operator theory, we conclude that these properties are determined by the interaction of the weight $\lambda$ and the symbol $\varphi$. In Section 2, we include some definitions and lemmas, which are needed for our further research. In Section 3, we consider  the boundedness of the composition operator $C_{\varphi}$ on $L^{p}_{\lambda}(T)$, and establish  a necessary and sufficient condition for $C_{\varphi}$ to be an isometry. In Section 4, we discuss the compactness of composition operators and obtain an equivalent condition for $C_{\varphi}$ to be compact. In the end, we study the Schatten $p$-class of the Hilbert space $L_{\lambda}^{2}(T)$. We will explore this problem in three cases: $1\leq p<2$, $p=2$ and $p>2$, and give some necessary conditions and sufficient conditions on the symbol $\varphi$ and the weight $\{\lambda(v)\}_{v\in T}$ such that $C_{\varphi}$ belongs to $S_{p}$ in each case.

\section{Preliminary and Lemmas}\label{S2}

Before we begin, we need to give the elementary definitions and notations about graphs, we refer interested readers to \cite{Al13} for more detailed information. A graph $G=(V,E)$ contains a countably-infinite set $V$ and a subset $E$ of $V\times V$, which are called the set of vertices of $G$ and the set of edges of $G$, respectively. In usual, we do not distinguish between a graph and its vertices, we write $v\in G$ instead of $v\in V$. Two vertices in a graph are neighbors if $(u, v)\in E$ or $(v, u)\in E$ (denoted by $u\sim v$ or $v\sim u$), and the number of neighbours of a vertex $u$ is called the degree of $u$, which is denoted by $deg(u)$. We say a graph is locally finite if $deg(u)<\infty$ for each $u\in G$.\

Let $u, v$ be two vertices in $G$, if there exist $(n+1)$ distinct vertices satisfying $u=u_{0}\sim u_{1}\sim u_{2}\sim\cdots\sim u_{n}=v$, then we call this finite vertices a path of length $n$ joining $u$ and $v$. If this path containing an additional edge $(v, u)$, then it is called a cycle. A graph $G$ is connected if there exists a path between any two of its vertices. A tree $T$ is a locally finite and connected graph without cycles. That is to say, for each pair of vertices of $T$, we can find a unique path connecting them. The number of elements of the set $\{v\in T: (v, u)\in E\}$ is called the indegree of the vertex $u$. If $u\in T$ and its indegree is zero, we say it is a root of the tree $T$ and denote it by $o$. A tree $T$ having a fixed root is called a rooted tree.\

Let $T$ be a rooted tree, we denote the length of the unique path joining two vertices $u, v\in T$ by $d(u,v)$, and we use the notation $|v|$ to denote the distance between the vertex $v$ and the root $o$. It is easy to see that a tree $T$ is a metric space with this length. A terminal of $T$ is a vertex of $T$ which has only one neighbour. In this article, all the trees we discuss will be rooted at $o$ and have no terminal vertices. Without special instruction, all the tree in the following context are of this form.\

When we mention a function defined on a tree $T$, we mean it is a function defined on its vertices which maps vertices of the tree $T$ into the complex plane $\mathbb{C}$. Recall the definitions of $L^{p}$ space and the weighted $L^{p}$ space of a tree given in \cite[Definitions 2.5]{Al14}, we give the definitions of the spaces we are dealing with.

\begin{defi}
Let $1\leq p <\infty$ and $T$ be a tree. We denote by $L^{p}(T)$ the space of complex-valued functions $f:V\rightarrow \mathbb{C}$ satisfying $$\sum \limits_{v\in T}|f(v)|^{p}<\infty.$$
\end{defi}
In contrast to the Bergman space $A^{p}$ defined on the unit disk of complex plane, $L^{p}(T)$ is a Banach space with the norm $\|f\|_{p}=\Big(\sum \limits_{v\in T}|f(v)|^{p}\Big)^{\frac{1}{p}}$. Given a tree $T$, suppose $\lambda$ is a nonnegative function on $T$, then we can regard the countable set $\{\lambda(v)\in \mathbb{R}_{+}:v\in T\}$, indexed by $T$, as a sequence and denote $\lambda$ by $\{\lambda(v)\}_{v\in T}$. We generalize $L^{p}(T)$ space into the weighed $L^{p}$ $(1\leq p<\infty)$ space of a tree in the next definition.
\begin{defi}
Let $1\leq p<\infty$ and $\{\lambda(v)\}_{v\in T}$ be a nonnegative function of a tree $T$, the weighted $L^{p}$ space on $T$, denoted by $L^{p}_{\lambda}(T)$, is defined as the collection of functions $f$ on $T$ satisfying $$\sum \limits_{v\in T}|f(v)|^{p}\lambda(v)<\infty.$$
\end{defi}
Let $f\in L^{p}_{\lambda}(T)$, we can endow it with the norm
$$\|f\|_{p}=\Big[\sum \limits_{v\in T}|f(v)|^{p}\lambda(v)\Big]^{\frac{1}{p}},$$ it turns out that $L^{p}_{\lambda}(T)$ is also a Banach space with this norm. In the case of $p=2$, $L^{2}_{\lambda}(T)$ is a separable Hilbert space with the inner product$$\langle f,g\rangle=\sum\limits_{v\in T}f(v)g(v)\lambda(v)$$
for any $f, g\in L^{2}_{\lambda}(T)$. Clearly, if the nonnegative function $\{\lambda(v)\}_{v\in T}$ is a constant, then $L^{p}_{\lambda}(T)$ is identical to $L^{p}(T)$.

There are several properties of $L^{p}_{\lambda}(T)$, we give them as the following lemmas, which are needed in the following sections.
\begin{lem}\label{L2.3}
For a fixed vertex $v$ in $T$, the normalized characterized function $f_{v}(w)=\frac{1}{\lambda^{\frac{1}{p}}(v)}\chi_{v}(w)$ is a unit vector of $L^{p}_{\lambda}(T)$. Moreover, let $\{v_{n}\}$ be a sequence of vertices in $T$ such that $|v_{n}|\rightarrow \infty$ as $n\rightarrow\infty$ and $f_{n}(w)=\frac{1}{\lambda^{\frac{1}{p}}(v_{n})}\chi_{v_{n}}(w)$ for each $n\in \mathbb{N}$, then $\{f_{n}\}$ is a sequence of functions which is bounded and converges to zero pointwisely.
\end{lem}
\begin{proof}
By the definition of the norm on $L^{p}_{\lambda}(T)$, we have $$\|f_{v}\|_{p}^{p}=\sum \limits_{w\in T}|f_{v}(w)|^{p}\lambda(w)=\sum \limits_{w\in T}\bigg|\frac{\chi_{v}(w)}{\lambda^{\frac{1}{p}}(v)}\bigg|^{p}\lambda(w)=\sum \limits_{w\in T}\frac{\chi_{v}(w)}{\lambda(v)}\cdot\lambda(w)=1,$$
thus $\|f_{v}\|_{p}=1$ and it is a unit vector in $L^{p}_{\lambda}(T)$.\

For each $n\in \mathbb{N}$, we have $\|f_{n}\|_{p}=1$, then $\{f_{n}\}$ is a bounded sequence in $L^{p}_{\lambda}(T)$. Let $v$ be a fixed vertex in $T$, we need to show that the sequence $\{f_{n}(v)\}$ converges to zero. Suppose $v\neq v_{n}$ for any $n\in \mathbb{N}$, then $f_{n}(v)=0$ for each $n\in \mathbb{N}$, clearly, $\{f_{n}(v)\}$ converges to zero. We may assume $v=v_{j}$ for some positive integer $j$, then for all $n>j$, we have $f_{n}(v)=0$, which implies $\{f_{n}(v)\}$ converges to zero in these vertices. Hence $f_{n}$ is a bounded sequence converging to zero pointwisely.
\end{proof}
The following lemma implies that point evaluation functionals on $L^{p}_{\lambda}(T)$ are bounded.
\begin{lem}\label{L2.4}
Let $1\leq p<\infty$, suppose that $v$ is a fixed vertex of $T$ and $K_{v}$ is a functional on $L^{p}_{\lambda}(T)$ defined by $K_{v}(f)=f(v)$ for each $f\in L^{p}_{\lambda}(T)$. Then $K_{v}$ is bounded and $\|K_{v}\|\leq \frac{1}{\lambda^{\frac{1}{p}}(v)}$.
\end{lem}
\begin{proof}
Noting that $$\|f\|_{p}^{p}=\sum\limits_{w\in T}|f(w)|^{p}\lambda(w)\geq |f(v)|^{p}\lambda(v)=|K_{v}(f)|^{p}\lambda(v),$$ thus $|K_{v}(f)|\leq \frac{1}{\lambda^{\frac{1}{p}}(v)}\|f\|_{p}$ for $f\in L^{p}_{\lambda}(T)$, it follows that $K_{v}$ is bounded and $\|K_{v}\|\leq \frac{1}{\lambda^{\frac{1}{p}}(v)}$, as desired.
\end{proof}
We will discuss the Schatten $p$-class of $L^{p}_{\lambda}(T)$ in the final part. It is necessary to introduce elementary definitions and lemmas about it, which can be seen in chapter 1 of \cite{Al15}. Let $H$ be a separable Hilbert space and $A$ be a compact operator of $H$, then we can find two orthonormal sets $\{e_{n}\}$ and $\{\sigma_{n}\}$ in $H$ so that
$$Ax=\sum\limits_{n=1}^{\infty}\mu_{n}\langle x,e_{n}\rangle \sigma_{n}, ~~x\in H.$$
The coefficient $\mu_{n}$ is determined by the compact operator $A$ and we call it the $n$-th singular value of $A$.

Let $1\leq p<\infty$, the Schatten $p$-class of $H$ (denoted by $S_{p}(H)$ or simply $S_{p}$), is defined to be the collection of all compact operators $A$ on $H$ such that $\sum\limits_{n=1}^{+\infty}|\mu_{n}|^{p}<\infty$, $i.e.$, its singular value sequence $\{\mu_{n}\}\in l^{p}(\mathbb{N})$. It can be seen that the space $S_{p}$ is a closed subspace of $l^{p}(\mathbb{N})$ and actually is a Banach space when it endows with the norm
$$\|A\|_{S_{p}}=\Big[\sum\limits_{n=1}^{\infty}|\mu_{n}|^{p}\Big]^{\frac{1}{p}},$$ where $\{\mu_{n}\}$ is the singular value sequence of $A$.

We point out, in usual, the space $S_{1}$ is called the trace class and $S_{2}$ is called the Hilbert-Schmidt class. More details are available in \cite{Al15}.\

Here, we list several lemmas which are used in Section 4. The following lemma comes from \cite[Theorem 1.22]{Al15} and gives an equivalent condition for a compact operator $A\in S_{2}$.
\begin{lem}\label{L2.5}
Let $H$ be a separable Hilbert space and $A$ be a compact operator with singular values $\{\mu_{n}\}$, then for any orthonormal basis $\{e_{n}\}$ of $H$, $$\sum\limits_{n=1}^{\infty}|\mu_{n}|^{2}=\sum\limits_{n=1}^{\infty}\|Ae_{n}\|^{2}=\sum\limits_{n,m=1}^{\infty}|\langle Ae_{n},e_{m}\rangle|^{2}.$$
\end{lem}
The following lemma derived from \cite[Theorem 1.24]{Al15}, it gives a necessary condition for a compact operator $A$ belongs to $S_{1}$ and  will be used in the Section 4.
\begin{lem}\label{L2.6}
Let $A\in S_{1}$ and $\{e_{n}\}$ be any orthonormal basis of $H$, then the series $\sum\limits_{k=1}^{\infty}\langle Ae_{k},e_{k}\rangle$ converges absolutely and the sum is independent of the choices of the orthonormal basis.
\end{lem}
The next two lemmas are from \cite[Proposition 1.29]{Al15} and \cite[Theorem 1.33]{Al15}. We will apply them to give a sufficient condition for $C_{\varphi}\in S_{p}$ $(1\leq p<2)$ and a necessary condition for $C_{\varphi}\in S_{p}$ $(p\geq2)$.
\begin{lem}\label{L2.7}
Let $1\leq p\leq2$ and $A$ be a compact operator on a separable Hilbert space $H$. Then $$\|A\|_{S_{p}}^{p}\leq\sum\limits_{n=1}^{\infty}\sum\limits_{k=1}^{\infty}|\langle Ae_{n},e_{k}\rangle|^{p}$$ for any orthonormal basis $\{e_{n}\}$ of $H$.
\end{lem}
\begin{lem}\label{L2.8}
Let $p\geq2$ and $A$ be a compact operator on $H$. Then $A$ is in $S_{p}$ if and only if $$\sum\limits_{n=1}^{\infty}\|Te_{n}\|^{p}<+\infty$$ for all orthonormal basis $\{e_{n}\}$ in $H$. Moreover, $$\|A\|_{p}=\sup\bigg\{\Big[\sum\limits_{n=1}^{\infty}\|Ae_{n}\|^{p}\Big]^{\frac{1}{p}}:\{e_{n}\} \ \rm{is~ the ~orthonormal~basis}\bigg\}.$$
\end{lem}

\section{Boundedness, Operator Norm and Isometry}\label{S3}

Let $T$ be a tree and $\varphi$ be a self-map of $T$. we define
$$\beta_{\varphi}=\sup\limits_{v\in T}\frac{\lambda(v)}{\lambda(\varphi(v))}.$$
The quantity $\beta_{\varphi}$ reveals the interaction between the weight $\lambda$ and the symbol $\varphi$, and it will be used in characterizing the boundedness of $C_{\varphi}$. In this section, beginning with the boundedness of composition operators, we give an equivalent condition for $C_{\varphi}$ to be bounded when $\varphi$ is injective and determine its operator norm, and then we proceed to explore a necessary and sufficient condition for $\C_{\varphi}$ to be an isometry.

\begin{thm}\label{U3.1}
Let $T$ be a tree and $\varphi$ be an injective self-map on $T$. Then $C_{\varphi}$ is bounded on $L^{p}_{\lambda}(T)$ if and only if $\beta_{\varphi}$ is finite. Furthermore, $\|C_{\varphi}\|=\beta^{\frac{1}{p}}_{\varphi}$.
\end{thm}
\begin{proof}
Assume that $C_{\varphi}$ is bounded. Given a vertex $w_{0}$ in $\varphi(T)$, we define a function $f_{w_{0}}$ by $f_{w_{0}}(v)=\frac{1}{\lambda^{\frac{1}{p}}(w_{0})}\chi_{w_{0}}(v)$ for $v$ in $T$, then $f_{w_{0}}\in L^{p}_{\lambda}(T)$ and $\|f_{w_{0}}\|_{p}=1$ by Lemma \ref{L2.3}. Since the self-map $\varphi$ is injective, there is a unique vertex $v_{0}$ in $T$ such that $\varphi(v_{o})=w_{0}$, then we have
\begin{align*}
\|C_{\varphi}f_{w_{0}}\|^{p}_{p}&=\sum\limits_{v\in T}|f_{w_{0}}(\varphi(v))|^{p}\lambda(v)\\
&=\sum\limits_{v\in T}\Big|\frac{\chi_{w_{0}}(\varphi(v))}{\lambda^{\frac{1}{p}}(w_{0})}\Big|^{p}\lambda(v)\\
&=\sum\limits_{v\in T}\frac{\chi_{w_{0}}(\varphi(v))\lambda(v)}{\lambda(w_{0})}\\
&=\sum\limits_{\varphi(v)=w_{0}}\frac{\lambda(v)}{\lambda(w_{0})}\\
&=\frac{\lambda(v_{0})}{\lambda(\varphi(v_{0}))},
\end{align*}
it follows that $\|C_{\varphi}\|^{p}\geq \|C_{\varphi}f_{w_{0}}\|_{p}^{p}=\frac{\lambda(v_{0})}{\lambda(\varphi(v_{0}))}$ for arbitrary of $v_{0}\in T$, we obtain that
$$\beta_{\varphi}=\sup\limits_{v\in T}\frac{\lambda(v)}{\lambda(\varphi(v))}\leq \|C_{\varphi}\|^{p}<\infty,$$ thus $\beta_{\varphi}$ is finite and $\|C_{\varphi}\|\geq\beta^{\frac{1}{p}}_{\varphi}$.

Conversely, we assume that $\beta_{\varphi}$ is finite. For any vector $f\in L^{p}_{\lambda}(T)$, we observe that
\begin{align*}
\|C_{\varphi}f\|^{p}_{p}&=\sum\limits_{v\in T}|f(\varphi(v))|^{p}\lambda(v)\\
&=\sum\limits_{v\in T}|f(\varphi(v))|^{p}\lambda(\varphi(v))\frac{\lambda(v)}{\lambda(\varphi(v))}\\
&\leq\sup\limits_{v\in T}\frac{\lambda(v)}{\lambda(\varphi(v))}\sum\limits_{v\in T}|f(\varphi(v))|^{p}\lambda(\varphi(v))\\
&\leq\sup\limits_{v\in T}\frac{\lambda(v)}{\lambda(\varphi(v))}\sum\limits_{w\in T}|f(w))|^{p}\lambda(w)\\
&=\beta_{\varphi}\|f\|^{p}_{p},
\end{align*}
then we have $C_{\varphi}$ is bounded and $\|C_{\varphi}\|^{p}\leq\beta_{\varphi}$. We conclude that $C_{\varphi}$ is bounded if and only if $\beta_{\varphi}$ is finite when the symbol $\varphi$ is injective, moreover, $\|C_{\varphi}\|=\beta_{\varphi}^{\frac{1}{p}}$.
\end{proof}
Using Theorem \ref{U3.1}, we immediately get the following corollary.
\begin{cor}
Suppose $T$ is a tree and the weight $\lambda$ is a constant. Then the composition operator $C_{\varphi}$ induced by injective self-map $\varphi$ is bounded on $L^{p}(T)$. Moreover, $\|C_{\varphi}\|=1.$
\end{cor}
Motivated by the preceding corollary, one may question whether every injective self-map induces a bounded composition operator. Before continuing our research, we consider the following example.
\begin{ex}
Suppose the weighted function $\lambda$ is defined by $\lambda(v)=\frac{1}{1+|v|}$. It is evident that $\lambda$ tends to zero as $|v|$ tend to infinity. We define an injective self-map $\varphi$ mapping a vertex $v$ to a vertex with length $|v|^{2}$, this can be done since the number of elements in the subset $\{w\in T: |w|=|v|^{2}\}$ is larger than the subset $\{w\in T: |w|=|v|\}$. Choose a sequence of vertex $\{v_{n}\}$ satisfying that $|v_{n}|$ tends to infinity as $n$ tends to infinity. Since
$\frac{\lambda(v_{n})}{\lambda(\varphi(v_{n}))}=\frac{1+|\varphi(v_{n})|}{1+|v_{n}|}=\frac{1+|v_{n}|^{2}}{1+|v_{n}|}\rightarrow\infty$ as $n\rightarrow\infty$, thus $C_{\varphi}$ is not bounded by Theorem \ref{U3.1}.
\end{ex}
If we impose some condition on the weight function $\lambda$, then the result will be valid. In actual, we have the following theorem.
\begin{thm}
Let $T$ be a tree and $\varphi$ be any self-map on it. Then every composition operator $C_{\varphi}$ induced by injective self-map $\varphi$ is bounded if and only if the weight function $\lambda$ is bounded and bounded away from zero, i.e., there are two positive numbers $m$ and $M$ such that $m\leq \lambda(v)\leq M$ for all $v\in T$.
\end{thm}
\begin{proof}
Firstly, we assume that there are two positive constants $m$ and $M$ such that $m\leq \lambda(v)\leq M$ for any $v\in T$. Then for any injective self-map $\varphi$ on $T$, we have
$$\beta_{\varphi}=\sum\limits_{v\in T}\frac{\lambda(v)}{\lambda(\varphi(v))}\leq\frac{M}{m}.$$ It follows that $\beta_{\varphi}$ is finite and $C_{\varphi}$ is bounded by Theorem \ref{U3.1}.

Next, we suppose $\beta_{\varphi}$ is finite for every injective self-map $\varphi$, we need to show that $\lambda(v)$ is both bounded and bounded away from zero for all $v\in T$.
Assume that $\lambda$ is unbounded, then exist a sequence $\{v_{n}\}$ of vertices such that $\lambda(v_{n})\rightarrow \infty$ as $n\rightarrow \infty$. Extract a subsequence of $\{v_{n}\}$, say $\{w_{n}\}$, satisfying $\lambda(w_{n})>\lambda^{2}(v_{n})$ for all positive integer $n$. By the countability of those sequences, we can define an injection $\psi_{1}$ on $T$ with $\psi_{1}(w_{n})=v_{n}$ for $n\in \mathbb{N}$. Then $$\frac{\lambda(w_{n})}{\lambda(\psi_{1}(w_{n}))}=\frac{\lambda(w_{n})}{\lambda(v_{n})}>\lambda(v_{n})\rightarrow\infty$$as $n\rightarrow \infty$. This implies that $\beta_{\psi_{1}}$ must be infinite, the contradiction shows that $\lambda(v)$ is bounded for all $v\in T$.
And then we suppose $\lambda$ is not bounded away from zero, then we can choose a sequence $\{v_{n}'\}$ of vertices such that $\lambda(v_{n}')\rightarrow 0$ as $n\rightarrow \infty$. Similarly, we can select a subsequence of $\{v_{n}'\}$, which we denote $\{u_{n}\}$, such that $\lambda(u_{n})<\lambda^{2}(v_{n}')$ for all positive integer $n$. We define an injection map $\psi_{2}$ on $T$ with $\psi_{2}(v_{n}')=u_{n}$. Observe that $$\frac{\lambda(v_{n}')}{\lambda(\psi_{2}(v_{n}'))}=\frac{\lambda(v_{n}')}{\lambda(u_{n})}>\frac{\lambda(v_{n}')}
{\lambda^{2}(v_{n}')}=\frac{1}{\lambda(v_{n}')}\rightarrow\infty$$ as $n\rightarrow \infty$, which contradicts to our assumption that $\beta_{\psi_{2}}=\sup\limits_{v\in T}\frac{\lambda(v)}{\lambda(\varphi(v))}$ is finite, thus $\lambda$ is bounded away from zero. This completes our proof of the theorem.
\end{proof}
We denote the number of element in a set $A$ by $|A|.$ In fact, we can reduce the requirement that $\varphi$ is injection with $|\{\varphi^{-1}(v)\}|\leq M$ for all $v\in \varphi(T)$. In Theorem \ref{U3.1}, we show that the conclusion is still valid in this case. In spite of it, we can not determine the norm of $C_{\varphi}$ as before, but only give an estimate of its norm when $\beta_{\varphi}$ is finite.
\begin{thm}
Let $T$ be a tree and $\varphi$ be a self-map of $T$ with $|\{\varphi^{-1}(v)\}|\leq M$ for every $v\in \varphi(T)$. Then the composition operator $C_{\varphi}$ induced by the symbol $\varphi$ is bounded on $L^{p}_{\lambda}(T)$ $(1\leq p<\infty)$ if and only if $\beta_{\varphi}$ is finite. Moreover, $\|C_{\varphi}\|\leq (M\beta_{\varphi})^{\frac{1}{p}}$.
\end{thm}
\begin{proof}
Suppose $C_{\varphi}$ is bounded. For any $w\in \varphi(T)$, we define $f(v)=\frac{\chi_{w}(v)}{\lambda^{\frac{1}{p}}(w)}$, then $f\in L^{p}_{\lambda}(T)$ and $\|f\|_{p}=1$ by Lemma \ref{L2.3}. Then
\begin{align*}
\|C_{\varphi}f\|^{p}_{p}&=\sum\limits_{v\in T}|f(\varphi(v))|^{p}\lambda(v)\\
&=\sum\limits_{v\in T}\Big|\frac{\chi_{w}(\varphi(v))}{\lambda^{\frac{1}{p}}(w)}\Big|^{p}\lambda(v)\\
&=\sum\limits_{v\in T}\frac{\chi_{w}(\varphi(v))\lambda(v)}{\lambda(w)}\\
&=\sum\limits_{\varphi(v)=w}\frac{\lambda(v)}{\lambda(w)}\geq \frac{\lambda(v)}{\lambda(\varphi(v))}
\end{align*}
for any $v\in \varphi^{-1}(w)$. It follows that $\|C_{\varphi}\|=\sup\limits_{\|f\|_{p}=1}\{\|C_{\varphi}f\|_{p}: f\in L^{p}_{\lambda}(T)\}\geq \big(\frac{\lambda(v)}{\lambda(\varphi(v))}\big)^{\frac{1}{p}}$ for $v\in \varphi^{-1}(w)$. Notice that as $w$ run over $\varphi(T)$, $\varphi^{-1}(w)$ run over the tree $T$, thus $\|C_{\varphi}\|\geq \big(\frac{\lambda(v)}{\lambda(\varphi(v))}\big)^{\frac{1}{p}}$ for all $v\in T$, which implies $\beta_{\varphi}=\sup\limits_{v\in T}\frac{\lambda(v)}{\lambda(\varphi(v))}$ is finite if $C_{\varphi}$ is bounded on $L^{p}_{\lambda}(T)$. 

Conversely, assume $\beta_{\varphi}$ is finite. For any $f\in L^{p}_{\lambda}(T)$,
\begin{align*}
\|C_{\varphi}f\|^{p}_{p}&=\sum\limits_{v\in T}|f(\varphi(v))|^{p}\lambda(v)\\
&=\sum\limits_{v\in T}|f(\varphi(v))|^{p}\lambda(\varphi(v))\frac{\lambda(v)}{\lambda(\varphi(v))}\\
&\leq \beta_{\varphi}\sum\limits_{v\in T}|f(\varphi(v))|^{p}\lambda(\varphi(v))\leq M\beta_{\varphi}\|f\|_{p}^{p},
\end{align*}
thus $\|C_{\varphi}\|\leq (M\beta_{\varphi})^{\frac{1}{p}}$, $i.e.$, $C_{\varphi}$ is bounded if $\beta_{\varphi}$ is finite. This completes our proof.
\end{proof}
In the rest of this section, we investigate the isometric properties for composition operators on $L^{p}_{\lambda}(T)$. We give an equivalent condition such that $C_{\varphi}$ to be an isometry by applying Theorem \ref{U3.1}.
\begin{thm}\label{U3.5}
Let $T$ be a tree and $\varphi$ be an injective self-map of $T$. For $1\leq p<\infty$, $C_{\varphi}$ is an isometry on $L^{p}_{\lambda}(T)$ if and only if $\varphi$ is surjective and $\frac{\lambda(v)}{\lambda(\varphi(v))}=1$ for any $v\in T$.
\end{thm}
\begin{proof}
Assume $C_{\varphi}$ is an isometry on $L^{p}_{\lambda}(T)$, thus $\|C_{\varphi}\|=1.$ Firstly, we show that $\frac{\lambda(v)}{\lambda(\varphi(v))}=1$ for any $v\in T$. By Theorem \ref{U3.1}, we have $\beta_{\varphi}=\sup\limits_{v\in T}\frac{\lambda(v)}{\lambda(\varphi(v))}=\|C_{\varphi}\|^{p}=1$. Suppose there exist some $u\in T$ such that $\frac{\lambda(u)}{\lambda(\varphi(u))}<1$, we define a function $f(v)=\frac{\lambda^{\frac{1}{p}}(u)}{\lambda^{\frac{1}{p}}(v)}\chi_{\varphi(u)}(v)$, then
$$\|f\|^{p}_{p}=\sum\limits_{v\in T}\bigg|\frac{\lambda^{\frac{1}{p}}(u)}{\lambda^{\frac{1}{p}}(v)}\chi_{\varphi(u)}(v)\bigg|^{p}\lambda(v)
=\sum\limits_{v\in T}\lambda(u)\chi_{\varphi(u)}(v)=\lambda(u).$$
Since $\varphi$ is injective and there only exists one vertex $v$ such that $v=\varphi(u)$. Applying Theorem \ref{U3.1}, we have $\lambda(u)$ is bounded, hence $f\in L^{p}_{\lambda}(T)$. Observe that
\begin{align*}
\|C_{\varphi}f\|^{p}_{p}&=\sum\limits_{v\in T}|f(\varphi(v))|^{p}\lambda(v)\\
&=\sum\limits_{v\in T}\bigg|\frac{\lambda^{\frac{1}{p}}(u)}{\lambda^{\frac{1}{p}}(\varphi(v))}\chi_{\varphi(u)}(\varphi(v))\bigg|^{p}\lambda(v)\\
&=\sum\limits_{v\in T}\frac{\lambda(u)}{\lambda(\varphi(v))}\chi_{\varphi(u)}(\varphi(v))\lambda(v)\\
&=\frac{\lambda(u)}{\lambda(\varphi(u))}\lambda(u)<\lambda(u).
\end{align*}
The last equality holds for the reason that $\varphi$ is injective and $\varphi(u)=\varphi(v)$ if and only if $u=v$. This contradicts that $C_{\varphi}$ is an isometry, thus $\frac{\lambda(v)}{\lambda(\varphi(v))}=1$ for all $v\in T$.

If $\varphi$ is not surjective, then we can find a vertex $w\notin \varphi(T)$. For this $w$, we define  $g(v)=\frac{\chi_{w}(v)}{\lambda^{\frac{1}{p}}(w)},$ then $g\in L^{p}_{\lambda}(T)$ and $\|g\|_{p}=1$ by Lemma \ref{L2.3}. By the assumption that $C_{\varphi}$ is isometric, we have $\|C_{\varphi}g\|=\|g\|_{p}=1$. On the other hand, we find
\begin{align*}
\|C_{\varphi}g\|^{p}_{p}&=\sum\limits_{v\in T}|g(\varphi(v))|^{p}\lambda(v)\\
&=\sum\limits_{v\in T}\Big|\frac{\chi_{w}(\varphi(v))}{\lambda^{\frac{1}{p}}(w)}\Big|^{p}\lambda(v)\\
&=\sum\limits_{v\in T}\frac{\chi_{w}(\varphi(v))\lambda(v)}{\lambda(w)}=0,
\end{align*}
since there is no vertex $v$ such that $w=\varphi(v)$. This is a contradiction, hence $\varphi$ is surjection.

Conversely, if $\varphi$ is surjection and $\frac{\lambda(v)}{\lambda(\varphi(v))}=1$ for all $v\in T$, then $\varphi$ is a bijection of $T$. For any $f\in L^{p}_{\lambda}(T)$, we have
\begin{align*}
\|C_{\varphi}f\|^{p}_{p}&=\sum\limits_{v\in T}|f(\varphi(v))|^{p}\lambda(v)\\
&=\sum\limits_{v\in T}|f(\varphi(v))|^{p}\lambda(\varphi(v))\frac{\lambda(v)}{\lambda(\varphi(v))}\\
&=\sum\limits_{v\in T}|f(\varphi(v))|^{p}\lambda(\varphi(v))\\
&=\sum\limits_{w\in T}|f(w)|^{p}\lambda(w)=\|f\|_{p}^{p}.
\end{align*}
It follows that $C_{\varphi}$ is an isometry on $L^{p}_{\lambda}(T)$, which completes the proof.
\end{proof}
As a result of Theorem \ref{U3.5}, we immediately get that $C_{\varphi}$ is an isometry on $L^{p}(T)$ if and only if $\varphi$ is a bijection on $T$.

\section{Compactness and Schatten class composition operator on $L^{2}_{\lambda}(T)$}\label{S4}

In this section, we investigate the compactness of composition operators $C_{\varphi}$ on $L^{p}_{\lambda}(T)$ $(1\leq p<\infty)$, and give an equivalent condition for $C_{\varphi}$ to be compact. In the case of $p=2$, $L^{2}_{\lambda}(T)$ is a Hilbert space, we further study the Schatten $p$-class composition operators on $L^{2}_{\lambda}(T)$ when $C_{\varphi}$ is compact on $L^{2}_{\lambda}(T)$.

Let $X, Y$ be two Banach spaces and $A:X\rightarrow Y$ be a linear bounded operator from $X$ to $Y$. If for any bounded sequence $\{x_{n}\}_{n=1}^{\infty}$ in $X$, then $\{Ax_{n}\}_{n=1}^{\infty}$ has a subsequence converges in $Y$, we say $A$ is a compact operator. In our case, $X$ and $Y$ are both the Banach space $L^{p}_{\lambda}(T)$ and $A$ is a composition operator $C_{\varphi}$ with the symbol $\varphi$. In order to characterize the compactness of $C_{\varphi}$, we define a function $f^{(n)}\in L^{p}_{\lambda}(T)$ by
\begin{align*}
f^{(n)}(v)=
\begin{cases}
f(v), & \mathrm{if} \ \ \ |v|\leq n, \vspace{2mm}\\
0, & \mathrm{if} \ \ \ |v|>n.
\end{cases}
\end{align*}
And then we define a linear operator $A_{n}$ by $A_{n}f=f^{(n)}$ for each $f\in L^{p}_{\lambda}(T)$. Clearly, each $A_{n}$ can be regarded as a finite rank projection on $L^{p}_{\lambda}(T)$ and its range $A_{n}L^{p}_{\lambda}(T)$ is a closed subspace of $L^{p}_{\lambda}(T)$, which can be represented by linearly independent unit vectors $\Big\{\frac{\chi_{v}}{\lambda{\frac{1}{p}}(v)}: |v|\leq n\Big\}$. For more properties of the sequence of operators $\{A_{n}\}$, we give it as a lemma.
\begin{lem}\label{L4.1}
For each $n\in \mathbb{N}$ and the operator $A_{n}$ given above, we have $A_{n}$ is compact on $L^{p}_{\lambda}(T)$. Moreover, $\|A_{n}\|\leq 1$ and $\|I-A_{n}\|\leq 1$.
\end{lem}
\begin{proof}
Let $f\in L^{p}_{\lambda}(T)$, then $\sum\limits_{v\in T}|f(v)|^{p}\lambda(v)=\|f\|^{p}_{p}<\infty$. Observe that $$\|A_{n}f\|_{p}^{p}=\sum\limits_{|v|\leq n }|f(v)|^{p}\lambda(v)\leq \sum\limits_{v\in T }|f(v)|^{p}\lambda(v)=\|f\|_{p}^{p}$$ and $$\|(I-A_{n})f\|_{p}^{p}=\sum\limits_{|v|>n }|f(v)|^{p}\lambda(v)\leq \sum\limits_{v\in T }|f(v)|^{p}\lambda(v)=\|f\|_{p}^{p},$$ we immediately obtain that $\|A_{n}\|\leq 1$ and $\|I-A_{n}\|\leq 1$.\

Let $\{f_{k}\}$ be a sequence of bounded vectors in $L^{p}_{\lambda}(T)$, then $\{A_{n}f_{k}\}$ is also bounded on $L^{p}_{\lambda}(T)$, hence it is bounded on $A_{n}L^{p}_{\lambda}(T)$. Since $A_{n}L^{p}_{\lambda}(T)$ is closed and finite dimensional, we can choose a sequence of $\{A_{n}f_{k}\}$ which converges to a vector in $A_{n}L^{p}_{\lambda}(T)$ by the Bolzano-Weierstrass theorem, which implies $A_{n}$ is compact on $L^{p}_{\lambda}(T)$ for each $n\in \mathbb{N}$.
\end{proof}
Let $A$ be a linear bounded operator on a Banach space $X$. If $A$ can be approximated by compact operators in the operator norm, then $A$ is also compact. This result can be found in a standard functional text. Based on this conclusion, we have the following theorem.
\begin{thm}\label{U4.2}
Let $T$ be a tree and $\varphi$ be an injective self-map on $T$ such that $C_{\varphi}$ is bounded on $L^{p}_{\lambda}(T)$. If $\lim\limits_{N\rightarrow\infty}\sup\limits_{|\varphi(v)|\geq N}\frac{\lambda(v)}{\lambda(\varphi(v))}=0$, then $C_{\varphi}$ is compact.
\end{thm}
\begin{proof}
Since $\varphi$ is an injective map of the tree $T$, the range of $\varphi$ is infinite. For each $n\in \mathbb{N}$, we have $A_{n}$ is compact by Lemma \ref{L4.1}, then $C_{\varphi}A_{n}$ is also compact by the assumption $C_{\varphi}$ is bounded. Fix a positive integer $N$, we have
\begin{align*}
\|C_{\varphi}-C_{\varphi}A_{n}\|^{p}&=\sup\limits_{\|f\|_{p}\leq 1}\|(C_{\varphi}-C_{\varphi}A_{n})f\|^{p}_{p}\\
&=\sup\limits_{\|f\|_{p}\leq 1}\sum\limits_{v\in T}|C_{\varphi}(I-A_{n})f(v)|^{p}\lambda(v)\\
&=\sup\limits_{\|f\|_{p}\leq 1}\sum\limits_{v\in T}|C_{\varphi}(I-A_{n})f(v)|^{p}\lambda(\varphi(v))\frac{\lambda(v)}{\lambda(\varphi(v))}\\
&\leq\sup\limits_{\|f\|_{p}\leq 1}\sum\limits_{|\varphi(v)|\geq N}|C_{\varphi}(I-A_{n})f(v)|^{p}\lambda(\varphi(v))\frac{\lambda(v)}{\lambda(\varphi(v))}\\
&+\sup\limits_{\|f\|_{p}\leq 1}\sum\limits_{|\varphi(v)|< N}|C_{\varphi}(I-A_{n})f(v)|^{p}\lambda(\varphi(v))\frac{\lambda(v)}{\lambda(\varphi(v))}\\
&\leq\sup\limits_{|\varphi(v)|\geq N}\frac{\lambda(v)}{\lambda(\varphi(v))}\sup\limits_{\|f\|_{p}\leq 1}\sum\limits_{|\varphi(v)|\geq N}|C_{\varphi}(I-A_{n})f(v)|^{p}\lambda(\varphi(v))\\
&+\sup\limits_{|\varphi(v)|< N}\frac{\lambda(v)}{\lambda(\varphi(v))}\sup\limits_{\|f\|_{p}\leq 1}\sum\limits_{|\varphi(v)|< N}|C_{\varphi}(I-A_{n})f(v)|^{p}\lambda(\varphi(v)).
\end{align*}
Note that
\begin{align*}
&\sup\limits_{\|f\|_{p}\leq 1}\sum\limits_{|\varphi(v)|\geq N}|C_{\varphi}(I-A_{n})f(v)|^{p}\lambda(\varphi(v))=\sup\limits_{\|f\|_{p}\leq 1}\sum\limits_{|w|\geq N}|(I-A_{n})f(w)|^{p}\lambda(w)\\
&\leq \sup\limits_{\|f\|_{p}\leq 1}\sum\limits_{w\in T}|(I-A_{n})f(w)|^{p}\lambda(w)=\|I-A_{n}\|\leq 1
\end{align*}
and $$\sum\limits_{|\varphi(v)|< N}|C_{\varphi}(I-A_{n})f(v)|^{p}\lambda(\varphi(v))=\sum\limits_{|w|< N}|(I-A_{n})f(w)|^{p}\lambda(w)=0$$ if $n>N$, since $(I-A_{n})f(v)=0$ if $|v|<N$ and $n>N$. Thus for any $N\in \mathbb{N}$, we have $$\|C_{\varphi}-C_{\varphi}A_{n}\|^{p}\leq\sup\limits_{|\varphi(v)|\geq N}\frac{\lambda(v)}{\lambda(\varphi(v))}$$ if each $n>N$. If $\lim\limits_{N\rightarrow\infty}\sup\limits_{|\varphi(v)|\geq N}\frac{\lambda(v)}{\lambda(\varphi(v))}=0$, then $C_{\varphi}$ can be approximated by a sequence of compact operators $\{C_{\varphi}A_{n}\}$, hence $C_{\varphi}$ is compact. It finishes our proof.
\end{proof}
In order to get converse result, we need the following lemma.
\begin{lem}\label{U4.3}
Let $T$ be a tree and $A$ be a compact operator on $L_{\lambda}^{p}(T)$. Given a bounded sequence $\{f_{n}\}$ in $L_{\lambda}^{p}(T)$ such that $f_{n}\rightarrow 0$ as $n\rightarrow \infty$ pointwisely, then the sequence $\{Af_{n}\}$ converges to zero in the norm of $L_{\lambda}^{p}(T)$.
\end{lem}
\begin{proof}
Suppose not, then exists an $\varepsilon_{0}>0$ and a subsequence of $\{f_{n}\}$, say $\{f_{n_{k}}\}$, such that $\|Af_{n_{k}}\|\geq \varepsilon_{0}$ for all $k=1, 2, \cdots.$ Since $A$ is compact and $\{f_{n_{k}}\}$ is bounded, we can find a subsequence $\{f_{n_{k_{j}}}\}$ of $\{f_{n_{k}}\}$ and $f\in L_{\lambda}^{p}(T)$ such that $\|Af_{n_{k_{j}}}-f\|_{p}\rightarrow 0$ as $j\rightarrow 0$. By assumption, $\{f_{n}\}$ converges to zero pointwisely, then $\{f_{n}\}$ converges to zero uniformly on compact subsets of $T$ for the reason that $T$ is a metric space, it follows that $\{f_{n_{k_{j}}}\}$ converges to zero uniformly on compact subsets of $T$. Since $A$ is compact and thus it is continuous, we have $\{Af_{n_{k_{j}}}\}$ converges to zero uniformly on compact subsets of $T$. Since every point evaluation functions on $L_{\lambda}^{p}(T)$ is bounded by Lemma \ref{L2.4}, therefore for any $v\in T$, $$|Af_{n_{k_{j}}}(v)-f(v)|\leq C\|Af_{n_{k_{j}}}-f\|_{p}$$ for some positive constant $C\geq 0$. Since $\|Af_{n_{k_{j}}}-f\|_{p}\rightarrow 0$ as $j\rightarrow \infty$, we have $\{Af_{n_{k_{j}}}\}$ converges to $f$ pointwisely. Note that $\{Af_{n_{k_{j}}}\}$ converges to zero uniformly on compact subsets of $T$, it yields that $f=0$, which contradicts to the assumption $\|Af_{n_{k_{j}}}\|\geq \varepsilon_{0}$ for all $k\in \mathbb{N}$. Hence the sequence $\{Af_{n}\}$ converges to zero in the norm of $L_{\lambda}^{p}(T)$.
\end{proof}
Actually, the lemma has been characterized in relevant articles, see \cite[Lemma 2.5]{Al11} and \cite[Lemma 4.1]{Al12}.
\begin{thm}\label{U4.4}
Let $T$ be a tree and $\varphi$ be an injective self-map on $T$ such that $C_{\varphi}$ is bounded. If $C_{\varphi}$ is compact on $L_{\lambda}^{p}(T)$, then $\lim\limits_{N\rightarrow \infty}\sup\limits_{|\varphi(v)|\geq N}\frac{\lambda(v)}{\lambda(\varphi(v))}=0.$
\end{thm}
\begin{proof}
Assume that $C_{\varphi}$ is compact. If $\lim\limits_{N\rightarrow \infty}\sup\limits_{|\varphi(v)|\geq N}\frac{\lambda(v)}{\lambda(\varphi(v))}>0$, then there is a positive constant $s$ so that $$\lim\limits_{N\rightarrow \infty}\sup\limits_{|\varphi(v)|\geq N}\frac{\lambda(v)}{\lambda(\varphi(v))}>s.$$
Then we can find a sequence of vertices $\{v_{n}\}$ with $|\varphi(v_{n})|\rightarrow \infty$ as $n\rightarrow \infty$ such that $\limsup\limits_{n\rightarrow\infty}\frac{\lambda(v_{n})}{\lambda(\varphi(v_{n}))}>s$. Let $f_{n}(v)=\frac{\chi_{\varphi(v_{n})}(v)}{\lambda^{\frac{1}{p}}(\varphi(v_{n}))}$, then $\{f_{n}\}$ is a sequence of unit vectors in $L_{\lambda}^{p}(T)$ which converges to zero pointwisely by Lemma \ref{L2.3}. Applying Lemma \ref{U4.3}, we obtain that $\|C_{\varphi}f_{n}\|_{p}\rightarrow 0$ as $n\rightarrow \infty$ since $C_{\varphi}$ is compact, then
\begin{align*}
s\geq\limsup\limits_{n\rightarrow\infty}\|C_{\varphi}f_{n}\|_{p}^{p}
&=\limsup\limits_{n\rightarrow\infty}\sum\limits_{v\in T}|f_{n}(\varphi(v))|^{p}\lambda(v)\\
&=\limsup\limits_{n\rightarrow\infty}\bigg(\sum\limits_{v\in T}\bigg|\frac{\chi_{\varphi(v_{n})}(\varphi(v))}{\lambda^{\frac{1}{p}}(\varphi(v_{n}))}\bigg|^{p}\lambda(v)\bigg)\\
&=\limsup\limits_{n\rightarrow\infty}\frac{\lambda(v_{n})}{\lambda(\varphi(v_{n}))}>s,
\end{align*}
which is a contradiction. Hence we must have $\lim\limits_{N\rightarrow \infty}\sup\limits_{|\varphi(v)|\geq N}\frac{\lambda(v)}{\lambda(\varphi(v))}=0$ if $C_{\varphi}$ is compact on $L_{\lambda}^{p}(T)$.
\end{proof}
Combining Theorem \ref{U4.2} and Theorem \ref{U4.4}, we give an equivalent condition for $C_{\varphi}$ to be compact. Let $\varphi$ is an injective self-map of a tree $T$ and $C_{\varphi}$ is bounded, then $C_{\varphi}$ is compact on $L_{\lambda}^{p}(T)$ if and only if $\lim\limits_{N\rightarrow \infty}\sup\limits_{|\varphi(v)|\geq N}\frac{\lambda(v)}{\lambda(\varphi(v))}=0$. Using this result, it is clear that the Banach space $L^{p}(T)$ has no compact composition operator $C_{\varphi}$ induced by injective self-map $\varphi$ of a tree $T$.\

In the following part, we further consider the Schatten $p$-class for the Hilbert space $L_{\lambda}^{2}(T)$. Let $p\geq 1$ and $S_{p}$ be the Schatten $p$-class of $L_{\lambda}^{2}(T)$. In a natural way, we want to search necessary and sufficient conditions on $\varphi$ such that the composition operator $C_{\varphi}$ belongs to $S_{p}$. We will discuss it in three cases: $1\leq p<2$, $p=2$ and $p>2$.\

Obviously, the vertices of a tree $T$ are countable, we define a sequence of functions $\{f_{u}\}_{u\in T}$ by
\begin{align*}
f_{u}(v)=\frac{\chi_{u}(v)}{\lambda^{\frac{1}{2}}(u)}=
\begin{cases}
\frac{1}{\lambda^{\frac{1}{2}}(u)}, & \mathrm{if} \ \ \ v=u, \vspace{2mm}\\
0, & \mathrm{if} \ \ \ v\neq u.
\end{cases}
\end{align*}
It is easy to verify that $\{f_{u}\}_{u\in T}$ forms an orthonormal basis of $L_{\lambda}^{2}(T)$. We firstly consider the case of $p=2$. Here we get an equivalent condition for the composition operator $C_{\varphi}$ belongs to $S_{2}$, which is usually called the Hilbert-Schmidt class.
\begin{thm}
Let $\varphi$ be an injective self-map of a tree $T$, and the composition operator $C_{\varphi}$ be a compact operator on $L_{\lambda}^{2}(T)$. Then $C_{\varphi}$ is Hilbert-Schmidt operator on $L_{\lambda}^{2}(T)$ if and only if $\sum\limits_{u\in T}\frac{\lambda(\varphi^{-1}(u))}{\lambda(u)}<\infty$. Moreover, we have $$\|C_{\varphi}\|_{S_{2}}=\Big[\sum\limits_{u\in T}\frac{\lambda(\varphi^{-1}(u))}{\lambda(u)}\Big]^{\frac{1}{2}}.$$
\end{thm}
\begin{proof}
By the definition of the norm and Lemma \ref{L2.5}, we have
\begin{align*}
\|C_{\varphi}\|_{S_{2}}^{2}&=\sum\limits_{u\in T}\|C_{\varphi}f_{u}\|_{2}^{2}=\sum\limits_{u\in T}\sum\limits_{v\in T}|C_{\varphi}f_{u}(v)|^{2}\lambda(v)\\
&=\sum\limits_{u\in T}\sum\limits_{v\in T}|f_{u}(\varphi(v))|^{2}\lambda(v)\\
&=\sum\limits_{u\in T}\sum\limits_{v\in T}\bigg|\frac{\chi_{u}(\varphi(v))}{\lambda^{\frac{1}{2}}(u)}\bigg|^{2}\lambda(v).
\end{align*}
Since $\varphi$ is injective, there exists at most one vertex $v$ such that $\varphi(v)=u$, thus
$$\|C_{\varphi}\|_{S_{2}}^{2}=\sum\limits_{u\in T}\Big(\sum\limits_{\varphi(v)=u}\frac{\lambda(v)}{\lambda(u)}\Big)=\sum\limits_{u\in T}\frac{\lambda(\varphi^{-1}(u))}{\lambda(u)}.$$
Therefore, $C_{\varphi}$ is a Hilbert-Schmidt operator on $L_{\lambda}^{2}(T)$ if and only if $\sum\limits_{u\in T}\frac{\lambda(\varphi^{-1}(u))}{\lambda(u)}<\infty$. Moreover, we have $\|C_{\varphi}\|_{S_{2}}=\Big[\sum\limits_{u\in T}\frac{\lambda(\varphi^{-1}(u))}{\lambda(u)}\Big]^{\frac{1}{2}}.$ This completes the proof.
\end{proof}
For the case of $p\neq 2$, it is difficult to establish an equivalent condition for $C_{\varphi}\in S_{p}$. In spite of this, we obtain a necessary condition for $C_{\varphi}\in S_{p}$ if $p> 2$, and a sufficient condition for $C_{\varphi}\in S_{p}$ if $0<p<2$, respectively.
\begin{thm}
Let $p>2$ and $\varphi$ be an injective self-map of the tree $T$ such that $C_{\varphi}$ is compact on $L_{\lambda}^{2}(T)$. If $C_{\varphi}\in S_{p}$ , then $$\sum\limits_{u\in T}\Big[\frac{\lambda(\varphi^{-1}(u))}{\lambda(u)}\Big]^{\frac{p}{2}}<\infty.$$
\end{thm}

\begin{proof}
Suppose $C_{\varphi}\in S_{p}$, then
\begin{align*}
\sum\limits_{u\in T}\|C_{\varphi}f_{u}\|_{2}^{p}&=\sum\limits_{u\in T}\Big[\sum\limits_{v\in T}|f_{u}(\varphi(v))|^{2}\lambda(v)\Big]^{\frac{p}{2}}\\
&=\sum\limits_{u\in T}\Big[\sum\limits_{v\in T}\Big|\frac{\chi_{u}(\varphi(v))}{\lambda^{\frac{1}{2}}(u)}\Big|^{2}\lambda(v)\Big]^{\frac{p}{2}}\\
&=\sum\limits_{u\in T}\Big[\sum\limits_{v\in T}\chi_{u}(\varphi(v))\frac{\lambda(v)}{\lambda(u)}\Big]^{\frac{p}{2}}\\
&=\sum\limits_{u\in T}\Big[\sum\limits_{\varphi(v)=u}\frac{\lambda(v)}{\lambda(u)}\Big]^{\frac{p}{2}}=\sum\limits_{u\in T}\Big[\frac{\lambda(\varphi^{-1}(u))}{\lambda(u)}\Big]^{\frac{p}{2}}.
\end{align*}
Using Lemma \ref{L2.8} and $C_{\varphi}\in S_{p}$, we have $\sum\limits_{u\in T}\|C_{\varphi}f_{u}\|_{2}^{p}=\sum\limits_{u\in T}\Big[\frac{\lambda(\varphi^{-1}(u))}{\lambda(u)}\Big]^{\frac{p}{2}}<\infty.$ As desired.
\end{proof}
\begin{thm}
Suppose $1\leq p<2$ and let $\varphi$ be an injective self-map of the tree $T$ such that $C_{\varphi}$ is compact on $L^{2}_{\lambda}(T)$. If $$\sum\limits_{u\in T}\Big[\frac{\lambda(\varphi^{-1}(u))}{\lambda(u)}\Big]^{\frac{p}{2}}<\infty,$$ then $C_{\varphi}\in S_{p}$.
\end{thm}
\begin{proof}
Assume $\sum\limits_{u\in T}\Big[\frac{\lambda(\varphi^{-1}(u))}{\lambda(u)}\Big]^{\frac{p}{2}}<\infty.$ Applying Lemma \ref{L2.7}, we have
\begin{align*}
\|C_{\varphi}\|_{S_{p}}^{p}&\leq\sum\limits_{u\in T}\sum\limits_{w\in T}\big|\langle C_{\varphi}f_{u},f_{w}\rangle\big|^{p}\\
&=\sum\limits_{u\in T}\sum\limits_{w\in T}\big|\sum\limits_{v\in T}f_{u}(\varphi(v))f_{w}(v)\lambda(v)\big|^{p}\\
&=\sum\limits_{u\in T}\sum\limits_{w\in T}\Big|\sum\limits_{v\in T}\frac{\chi_{u}(\varphi(v))}{\lambda^{\frac{1}{2}}(u)}\cdot\frac{\chi_{w}(v)}{\lambda^{\frac{1}{2}}(w)}\lambda(v)\Big|^{p}\\
&=\sum\limits_{u\in T}\sum\limits_{w\in T}\Big[\frac{\chi_{u}(\varphi(w))}{\lambda^{\frac{1}{2}}(u)}\lambda^{\frac{1}{2}}(w)\Big]^{p}\\
&=\sum\limits_{u\in T}\sum\limits_{\varphi(w)=u}\Big[\frac{\lambda(w)}{\lambda(u)}\Big]^{\frac{p}{2}}=\sum\limits_{u\in T}\Big[\frac{\lambda(\varphi^{-1}(u))}{\lambda(u)}\Big]^{\frac{p}{2}}<\infty.
\end{align*}
This gives that $C_{\varphi}\in S_{p}$ for $1\leq p<2$.
\end{proof}
In the final, we give a necessary condition for $C_{\varphi}$ to be a  trace operator.
\begin{cor}
Let $\varphi$ be an injective self-map of a tree $T$ and $C_{\varphi}$ be a compact  operator on $L_{\lambda}^{2}(T).$ If $C_{\varphi}\in S_{1},$ then $\varphi$ has a finite number of fixed points, $i.e.$, $$|\{u\in T:\varphi(u)=u\}|<\infty.$$
\end{cor}
\begin{proof}
Let $C_{\varphi}\in S_{1}$, we have by Lemma \ref{L2.6} that
\begin{align*}
\sum\limits_{u\in T}\langle C_{\varphi}f_{u},f_{u}\rangle&=\sum\limits_{u\in T}\Big(\sum\limits_{v\in T}f_{u}(\varphi(v))f_{u}(v)\lambda(v)\Big)\\
&=\sum\limits_{u\in T}\sum\limits_{v\in T}\frac{\chi_{u}(\varphi(v))}{\lambda^{\frac{1}{2}}(u)}\cdot\frac{\chi_{u}(v)}{\lambda^{\frac{1}{2}}(u)}  \lambda(v)\\
&=\sum\limits_{u\in T}\sum\limits_{v\in T}\chi_{u}(\varphi(v))\chi_{u}(v)\frac{\lambda(v)}{\lambda(u)}\\
&=\sum\limits_{u\in T}\chi_{u}(\varphi(u))=|\{u\in T:\varphi(u)=u\}|<\infty.
\end{align*}
This finishes the proof.
\end{proof}
\vspace{2.12mm}
\subsection*{Acknowledgment}
We are grateful to Professor Xianfeng Zhao (Chongqing University) for many useful discussions and suggestions.
This work was partially supported by  the Chongqing Natural Science Foundation (No. cstc2019jcyj-msxmX0337) and the Fundamental Research Funds for the Central Universities (grant numbers: 2020CDJQY-A039, 2020CDJ-LHSS-003).

\end{document}